\documentclass[english]{amsart}
\usepackage{amsmath,amssymb,amsthm,amsfonts,amscd,mathrsfs}

\usepackage{babel}
\usepackage[all]{xy}
\usepackage{tikz}
\usetikzlibrary{decorations.pathmorphing}
\usepackage{tikz-cd}
 \usepackage{graphicx} 
 \usepackage{epstopdf}
 \usepackage{pb-diagram}
 \usepackage[]{todonotes}
\usepackage{mathtools}
\usepackage{hyperref}

\newcommand{\bdm}{\begin{displaymath}}
\newcommand{\edm}{\end{displaymath}}

\newcommand{\Z}{\mathbb{Z}}

\newcommand{\co}{\colon\thinspace}

\newcommand{\X}{{\mathcal X}}

\newcommand{\cK}{{\mathcal K}}

\newcommand{\e}{\epsilon}
\newcommand{\cL}{\mathcal L}

\newcommand{\cH}{\mathcal H}

\newcommand{\Y}{{\mathcal Y}}

\newcommand{\id}{\operatorname{id}}
\newcommand{\Map}{\operatorname{Map}}

\newcommand{\ev}{\operatorname{{\it ev}}}
\newcommand{\cat}{\operatorname{cat}}

\newcommand{\gcat}{\operatorname{cat_G}}

\newcommand{\Gd}{{\mathcal G}}

\theoremstyle{definition}
\newtheorem{defn}{Definition}[section]

\newtheorem{exam}[defn]{Example}
\newtheorem{remark}[defn]{Remark}

\theoremstyle{plain}
\newtheorem{thm}[defn]{Theorem}
\newtheorem{prop}[defn]{Proposition}

\newtheorem{cor}[defn]{Corollary}

\begin{document}

\title[]{G-category versus orbifold category}
\author{A. Angel}
 \author{H. Colman}

\date{\today}
\dedicatory{In memory of  Edward Fadell and Sufian Husseini.}

\begin{abstract} We present a comparative study of certain invariants defined for group actions and their analogues defined for orbifolds. In particular, we prove that Fadell's equivariant category for $G$-spaces coincides with the Lusternik-Schnirelmann category for orbifolds when the group is finite. 
\end{abstract}

\maketitle
\section*{Foreword}
It is a great pleasure for us to dedicate this work to the memory of Edward Fadell and Sufian Husseini whose broad contributions to several fields have greatly influenced our own work. The first author's work in Topological Combinatorics is significantly rooted on the celebrated Fadell-Husseini index. The second author's PhD thesis was vastly inspired by Fadell's first papers on fixed point theory, equivariant maps and Lusternik-Schnirelmann invariants. Fadell and Husseini's seminal work on configuration spaces continues to inspire to this day not only her work but also her students' work on Topological Robotics.

\section*{Introduction}
All orbifolds may be described as global quotients of spaces by compact group actions with finite isotropy groups \cite{pardon}. But different group actions may describe the same orbifold. The notion of Morita equivalence provides the means to classify group actions: an orbifold will be determined by a Morita equivalence class of group actions.

In this paper, we propose a comparative study of invariants defined for group actions and their analogues for orbifolds. Some invariants can be adapted and extended to the orbifold setting, even though there are some important differences.

We have previously introduced in \cite{AC} the notion of {\em path orbifold} (or orbifold of paths) for  discrete group actions. We compare the equivariant classical notion of $G$-connectedness for $G$-spaces with the orbifold notion of connectedness derived from our orbifold of paths. Some $G$-spaces that are not $G$-connected, turn out to be connected when considered as orbifolds. 

The Lusternik-Schnirelmann category was introduced in the 1930's with the goal of providing a lower bound for the number of critical points for any smooth function on a compact manifold \cite{opreaSurvey}.  Presently the LS category is an important numerical invariant in algebraic topology, critical point theory, symplectic geometry and topological robotics. An equivariant version of  the LS category was introduced by Fadell in \cite{F} and later studied further in \cite{MARZ, Clapp, Colman2005, morita}. A {\em $G$-categorical} set in the sense of Fadell is one that can be compressed into an orbit via an equivariant homotopy. The $G$-category is the minimal number of $G$-categorical sets needed to cover a $G$-space. From our notion of path orbifold, we derive a notion of orbifold homotopy. Our main result establishes that the $G$-category coincides with the orbifold LS category for finite group actions. Our methods include the theory of Hilsum-Skandalis maps \cite{M} which allows us to state a clear relationship between 
  equivariant and orbifold maps, and provides a blueprint for similar comparisons of other invariants. 
 
This paper begins with two background sections covering the areas of $G$-spaces and topological groupoids. Section 1 gives the basic definitions in the category of $G$-spaces, consisting of spaces with an action
of the group $G$ and equivariant maps between them. Section 2 establishes the basic definitions in the bicategory of topological groupoids which includes the notion of Morita equivalence and generalized maps. We introduce here the concept of the path groupoid and derive the groupoid homotopy notion. A brief introduction to orbifolds is given in section 3. In section 4 we commence the comparative study of invariants for $G$-spaces and for orbifolds. We show here that the property of being connected is a point where the two theories diverge. In section 5 we present an adaptation of the theory of Hilsum-Skandalis maps for group actions and give an explicit condition for a generalized map between groupoids to be equivalent to a $G$-map between $G$-spaces. In the final section we continue to explore the comparison of invariants. We introduce the orbifold LS category derived from the orbifold homotopy notion and prove that the $G$-category introduced by Fadell for $G$-spaces coincides with our orbifold LS category when the group is finite.

\section{Background on $G$-spaces}

A {\em $G$-space} is a Hausdorff topological space $X$ equipped with a continuous action of a topological group $G$.

For each $x\in X$ the {\em isotropy group} $G_x=\{h\in G\mid  hx=x\}$ is a closed subgroup of $G$. The set $Gx=\{gx\mid g\in G\}\subseteq X$ is called the {\em orbit} of $x$. 

The {\em orbit space} $X/G$ is the set of equivalence classes determined by the action, endowed with the quotient topology. 

 If $H$ is a closed subgroup of $G$, then $X^H=\{x\in X|\; hx=x \mbox{ for all }h\in H\}$ is called the {\em $H$-fixed point set} of $X$. We call $x$ a {\em global fixed point} if $x \in X^G$.
 
 If $X$ and $Y$ are both $G$-spaces for the same group $G$, then a {\em $G$-map} is a map $f: X \to Y$ such that $f(gx) = gf(x)$
for all $g\in G$ and all $x\in X$.

\subsection{G-homotopy}

Let $X$ and $Y$ be $G$-spaces.
Two $G$-maps $\phi, \psi\co  X\to Y$ are {\em $G$-homotopic}, written $\phi\simeq_G \psi$, if there is a $G$-map $F\co X \times I \to Y$ with $F_0=\phi$ and $F_1=\psi$, where $G$ acts trivially on $I$ and diagonally on $X\times I$.

If there exist $G$-maps $\phi\co X\to Y$ and $\psi\co Y\to X$ such that $\phi\psi\simeq_G \id_Y$ and $\psi\phi\simeq_G \id_X$, then $\phi$ and $\psi$ are {\em $G$-homotopy equivalences}, and $X$ and $Y$ are {\em $G$-homotopy equivalent}, written $X\simeq_G Y$.

\section{Background on topological groupoids}
A {\em topological groupoid $\Gd$} given by $G_1\rightrightarrows G_0$ is an internal groupoid in the category of topological spaces; that is a groupoid with a topological space of objects $G_0$ and one of morphisms $G_1$ together with the usual structure maps:  source and target  $s, t\colon  G_1 \to G_0$, identity arrows determined by  $u\colon  G_0 \to G_1$,  inversion $i : G_1 \to  G_1$, and composition
$m\colon  G_1 \times_{s, G_0, t} G_1 \to G_1$, all given by continuous maps, such that $s$ (and therefore $t$)  is an open surjection.

Let $G$ be a topological group acting continuously on a Hausdorff space $X$. From this data we can construct a topological groupoid, the {\em translation groupoid} $G\ltimes X$, whose objects are the elements of $X$ and whose morphisms $x \to y$ are pairs $(g, x) \in G \times X$ such that $gx = y$, with composition induced by multiplication in $G$. That is, $G\ltimes X$ is the groupoid $G\times X\rightrightarrows X$ where the source is the second projection and the target map is given by the action.

An {\em equivariant map} $m\ltimes \e: K\ltimes Z\to G\ltimes X$ consists of a homomorphism $m: K \to G$ and a continuous map $\e: Z\to X$,  such that $\e(kz)=m(k)\e(z)$ for all $k\in K$, $z\in Z$.

\subsection{Natural transformations.}
Let $m\ltimes \phi, n\ltimes \psi: K\ltimes Z\to G\ltimes X$ be two equivariant maps. A {\em natural transformation} $\lambda$ from $\phi$
to $\psi$  is a map $\lambda:Z\to G\times X$ giving for
each $z\in Z$ an arrow $\lambda(z):\phi(z)\to \psi(z)$ in $G\times X$, such  that for any
$k$ in $K$ the identity $n(k)\lambda(z)=\lambda(kz)m(k)$ holds. We write $\phi\sim \psi$ if
such a $\lambda$ exists. 

\subsection{Morita equivalence}

An {\em essential equivalence} $m\ltimes \e: K\ltimes Z\to H\ltimes Y$ is an equivariant map satisfying:
\begin{enumerate}
\item(essentially surjective) $\phi'\circ \pi$ is an open surjection:
$$\xymatrix{Z\times_Y(H\times Y) \ar[r]^{\pi}  \ar[d]^{}& H \times Y \ar[r]^{\phi'}\ar[d]^{p_2}&Y \\ Z  \ar[r]^{\e}& Y& }$$
\item(fully faithful) the following diagram is a pullback:
$$\xymatrix{K\times Z\ar[d]^{(p_2, \phi)}  \ar[r]^{m\times \e}& H \times Y \ar[d]^{(p_2, \phi')} \\ Z\times Z  \ar[r]^{\e\times \e}& Y\times Y }$$
i.e. $K\times Z=\{((h,y),(z,z')) | y=\e(z), hy=\e(z')\}$.
\end{enumerate}

The first condition implies that for all $y\in Y$, there exists $z\in Z$ and $h\in H$ such that $\e(z)=hy$,
in other words an essential equivalence is not necessarily surjective but has to reach all of the orbits. The second assures that the isotropies are kept;
an essential equivalence cannot send points from different orbits in $Z$ to the same orbit  in $Y$ and there is a bijection between the sets:
$$\{k\in K| z'=kz\}=\{h\in H | \e(z')=h\e(z)\}.$$

Pronk and Scull showed a characterization of essential equivalences for group actions that can be used in practice to construct or check Morita equivalences.

\begin{prop}\cite{Pronk:2010} Any essential equivalence is a composite of maps of the forms (1) and (2) described below.
\begin{enumerate} 
\item\label{pronk1}(quotient map) $K\ltimes Z \to K/L\ltimes Z/L$ where $L$ is a normal subgroup of $G$ acting freely on $Z$.
\item\label{pronk2} (inclusion map) $K\ltimes Z \to  G\ltimes (G\times_KZ)  $ where $K$ is a (not necessarily normal) subgroup of $G$ acting on $Z$ and $G\times_KZ=G\times Z/\sim$ with $(gk^{-1},kz)\sim (g,z)$ for any $k \in K$.
\end{enumerate}
\end{prop}

\begin{remark} \label{pronk} An essential equivalence
$m\ltimes \e: K\ltimes Z\to H\ltimes Y$  can be factored as 
$$  K\ltimes Z\to   K/L\ltimes Z/L \to  H\ltimes (H\times_{K/L}{Z/L})  \cong  H\ltimes Y $$
where $L$ is a normal subgroup of $K$ acting freely on $Z$ and $K/L$ is a subgroup of $H$.
\end{remark}

The translation groupoids $G\ltimes X$ and $H\ltimes Y$ are {\em Morita equivalent} if there is a third groupoid $K\ltimes Z$ and two essential equivalences
$$\xymatrix{
{G\ltimes X}& {K\ltimes Z} \ar[r]^{m\ltimes \e} \ar[l]_{m'\ltimes \e'}& {H\ltimes Y}.
}$$
We write $G\ltimes X \sim_M H\ltimes Y$.

\subsection{Generalized maps}
A {\em generalized map} from $H\ltimes Y$ to $G\ltimes X$ is given by first replacing
$H\ltimes Y$ by a Morita equivalent groupoid $K\ltimes Z$ and then mapping $K\ltimes Z$ into $G\ltimes X$ by an equivariant map. 

We can formalize this  via the bicalculus of fractions \cite{Pronk1996}. Our context will be the Morita bicategory of translation groupoids \cite{Pronk:2010} obtained  by formally inverting the essential equivalences. Objects in this bicategory are translations groupoids,  $1$-morphisms are  generalized maps

$$H\ltimes Y\overset{m\ltimes \epsilon}{\gets}K\ltimes Z\overset{n\ltimes \phi}{\rightarrow}G\ltimes X$$
such that $m\ltimes \epsilon$ is an essential equivalence and 
2-morphisms from $H\ltimes Y\overset{m\ltimes \epsilon}{\gets}K\ltimes Z\overset{n\ltimes \phi}{\rightarrow}G\ltimes X$ to $H\ltimes Y\overset{m'\ltimes \epsilon'}{\gets}K'\ltimes Z'\overset{n\ltimes \phi}{\rightarrow}G\ltimes X$ are given by classes of diagrams:
$$
\xymatrix{ &
{K\ltimes Z}\ar[dr]^{n\ltimes \phi}="0" \ar[dl]_{m\ltimes \epsilon }="2"&\\
{H\ltimes Y}&{\cL} \ar[u]_{\nu} \ar[d]^{\mu}&{G\ltimes X}\\
&{K'\ltimes Z'}\ar[ul]^{m'\ltimes \epsilon'}="3" \ar[ur]_{n'\ltimes \phi'}="1"&
\ar@{}"0";"1"|(.4){\,}="7"
\ar@{}"0";"1"|(.6){\,}="8"
\ar@{}"7" ;"8"_{\sim}
\ar@{}"2";"3"|(.4){\,}="5"
\ar@{}"2";"3"|(.6){\,}="6"
\ar@{}"5" ;"6"^{\sim}
}
$$
where $\cL$ is a translation groupoid, and $\nu$ and $\mu$ are essential  equivalences.

We will see in this next section that the {\em path groupoid} of $G\ltimes X$ is defined as the mapping groupoid in this bicategory.

\subsection{Generalized paths}
\begin{defn}\cite{AC} A {\em generalized path} in the groupoid $G\ltimes X$ is a generalized map $I\overset{\e}{\gets}I'\overset{\alpha}{\to}G\ltimes X$.
\end{defn}

We will say that a groupoid is {\em developable} if it is Morita equivalent to a translation groupoid $G\ltimes X$ with $G$ a discrete group.

Consider the groupoid
$I_{S_n}$  associated to a subdivision $$S_n=\{0=r_0\le r_1<\cdots<r_{n-1}\le r_n=1\}$$ of the interval $I=[0,1]$ as explained below.

The space of objects of the groupoid $I_{S_n}$  is the disjoint union
$$\bigsqcup_{i=1}^{n} I_i$$
where $I_i$ is a small open neighborhood of $[r_{i-1}, r_{i}]$ and $(r,i)$ denotes an element $r$ in the connected component  $I_i$.

The space of arrows of $I_{S_n}$ is given by the disjoint union $$\left(\bigsqcup_{i=1}^{n}  I_i\right) \sqcup\left(\bigsqcup_{i=1}^{n-1} ( \tilde{I}_i \sqcup \tilde{I}_i)\right)$$

where $\displaystyle\bigsqcup_{i=1}^{n}  I_i$ is the set of unit arrows,
$ \tilde{I}_i= I_i\cap I_{i+1}$ and another copy $\tilde{I}_i$ was added for inverse arrows.

A  generalized path in a developable groupoid $G\ltimes X$ is equivalent to a generalized map $I\overset{\e}{\gets}I_{S_n}\overset{\alpha}{\to}G\ltimes X$ such that

\begin{enumerate}
\item
$\e: I_{S_n}\to I$ on objects is the inclusion on each connected component, $\e(r,i)=r$ and on arrows it sends all arrows to identity arrows, $\e(\tilde r_i)=\id_{r}$
\item 
$\alpha: I_{S_n}\to G\ltimes X$ on objects is given by a map $\alpha_i :  I_i\to X$ in each connected component and on arrows is given by 
$\alpha(\tilde r_i)=(k_i, \alpha_i( r))$
satisfying the condition $k_i  \alpha_i( r)=\alpha_{i+1}( r)$ for all $r\in \tilde I_i$.
\end{enumerate}

Notation: $\alpha=(\alpha_1,k_1, \alpha_1, \ldots,  k_{n-1}, \alpha_n)$.

Two generalized paths are equivalent if there exists a common subdivision $S_n$ and $g_i\in G$  such that $\beta_i(r)=g_i\alpha_i(r)$  and $k'_i=g_{i+1} k_i {g_i}^{-1}$ for all i.

$$
\xymatrix{ &
{I_{S_m}}\ar[dr]^{\alpha}="0" \ar[dl]_{\epsilon }="2"&\\
{I}&{I_{S_n}} \ar[u]_{a} \ar[d]^{b}&{G\ltimes X}\\
&{I_{S_{m'}}}\ar[ul]^{\tau}="3" \ar[ur]_{\beta}="1"&
\ar@{}"0";"1"|(.4){\,}="7"
\ar@{}"0";"1"|(.6){\,}="8"
\ar@{}"7" ;"8"_{\sim}
\ar@{}"2";"3"|(.4){\,}="5"
\ar@{}"2";"3"|(.6){\,}="6"
\ar@{}"5" ;"6"^{\sim}
}
$$

We defined the {\em path groupoid} of $G\ltimes X$ in \cite{AC} as $P=P(G\ltimes X)$ where the space of objects and arrows are given by the following colimits:

$$P_0= {\coprod  \Map( I_{S_n},G\ltimes X)}/\sim$$
and $$P_1= \left(\left({\coprod G^n}\right)/\sim \right)\times \left(\left({\coprod \Map( I_{S_n},G \ltimes X)}\right)/\sim\right)$$ 
\begin{thm}\cite{AC} Let $G$ be a discrete group. Then
$$P(G\ltimes X) \sim_M G\ltimes X^I$$
where $X^I$ is the path space of $X$.
\end{thm}

\subsection{grd-homotopy}
We use the previous characterization of the path groupoid to derive a notion of homotopy between generalized maps. All groups in this section are discrete.
\begin{defn}
Two generalized maps  $K\ltimes Y\overset{\sigma}{\gets}K'\ltimes Y'\overset{f}{\to}G\ltimes X$  and  $K\ltimes Y\overset{\tau}{\gets}K''\ltimes Y''\overset{g}{\to}G\ltimes X$ are {\em grd-homotopic} if there is a generalized map 
$K\ltimes Y\overset{\e}{\gets}\tilde K\ltimes \tilde Y\overset{H}{\to}G\ltimes X^I$ such that the following diagram commutes up to $2$-isomorphism:

$$
\xymatrix{ G\ltimes X&&
{G \ltimes X^I}\ar[rr]^{\ev_{1}}="0" \ar[ll]_{\ev_0 }="2"&&G\ltimes X\\
&{K'\ltimes Y'}\ar[ul]^{f}\ar[dr]_{\sigma}&{\tilde K\ltimes \tilde Y} \ar[u]_{H} \ar[d]^{\e}&{K''\ltimes Y''}\ar[ur]_{g}\ar[dl]^{\tau}&\\
&&{K\ltimes Y}&&
}
$$

\end{defn}
Notation $(\sigma,f)\simeq_{grd}(\tau, g)$. 

 Observe that if the generalized maps are equivariant maps ($\sigma=\tau=\id$) then the previous definition specializes to:
 \begin{remark}
$f\simeq_{grd} g$ if there exists a generalized map $(\e, H)$ and
essential equivalences $\eta$ and $\nu$ such that the following diagram commutes up to natural transformations.
$$
\xymatrix{ G\ltimes X&&
{G \ltimes X^I}\ar[rr]^{\ev_{1}} \ar[ll]_{\ev_0 }&&G\ltimes X\\
&{\bullet}\ar[r]^{\nu\;\;\;}&{\tilde K\ltimes \tilde Y} \ar[u]_{ H} \ar[d]^{\epsilon}&{\bullet}\ar[l]_{\;\;\;\eta}&\\
&&{K\ltimes Y}\ar@/^2.0pc/[uull]^{f}\ar@/_2.0pc/[uurr]_{g}&&
}
$$

\end{remark}

\section{Orbifolds}
We recall now the description of orbifolds as groupoids due to Moerdijk and Pronk \cite{Pronk1996, Pronk1997}. Orbifolds were first introduced by Satake \cite{S} as a generalization of a manifold defined in terms of local quotients. The groupoid approach provides a global language to reformulate the notion of orbifold.

 A groupoid $\Gd$ is {\it proper} if $(s,t):G_1\to G_0\times G_0$ is a proper map and it is a {\it foliation} groupoid if each isotropy group is discrete. 

\begin{defn}
An {\it orbifold} groupoid is a proper foliation groupoid.
\end{defn}

Given an orbifold groupoid $\Gd$, its orbit space $|\Gd|$ is a locally compact Hausdorff space. Given an arbitrary locally compact Hausdorff space $X$ we can equip it  with an orbifold structure as follows:

\begin{defn} An {\it orbifold structure} on a locally compact Hausdorff space $X$ is given by an orbifold groupoid $\Gd$ and a homeomorphism $h:|\Gd|\to X$.
\end{defn}

If $\e:\cH\to \Gd$ is an essential equivalence and $|\e|:|\cH|\to |\Gd|$ is the induced homeomorphism between orbit spaces, we say that the composition $h\circ|\e|:|\cH|\to X$ defines an {\it equivalent} orbifold structure. 

\begin{defn}  An {\it orbifold}  $\X$ 
is a space $X$ equipped with an equivalence class of orbifold 
structures. A specific such structure, given by 
$\Gd$ and  $h : |\Gd | \to X $ is
a {\it presentation} of the orbifold 
$\X$.
\end{defn}

If two groupoids are Morita equivalent, then they define the same orbifold. Therefore any structure or invariant for orbifolds, if defined through groupoids, should be invariant under Morita equivalence. 

\begin{defn}  An {\it orbifold map} $f\colon \Y\to \X$ is given by an equivalence class of generalized maps $(\e,\phi)$ from $\cK$ to $\Gd$ between presentations of the orbifolds such that the diagram commutes:

\[\xymatrix{
|\cK| \ar[r]^{|\phi||\e|^{-1}}\ar[d]& |\Gd| \ar[d]^{} \\ 
Y \ar[r]& X}\]
A specific such generalized map $(\e,\phi)$ is called a {\it presentation} of the orbifold map $f$.
\end{defn}

We can obtain an orbifold by considering the action of a compact group $G$ acting on a space $X$ with finite stabilizers. All orbifolds can be described in this way \cite{pardon}.

An orbifold $\X$ is {\em developable} is it is presented by a developable groupoid.

\begin{exam}[The $(p,q)$-action]

Consider $S^1\subset \mathbb C$ and $S^3=\{(z_1,z_2) \colon |z_1|^2+|z_2|^2=1\}\subset \mathbb C^2$. The 3-sphere $S^3$ can be decomposed by two solid tori $T_1=\{(z_1,z_2) \colon |z_1|\le \sqrt{2}/2\}$ and $T_2=\{(z_1,z_2) \colon |z_2|\le \sqrt{2}/2\}$.
Consider the action of  $S^1$ on $S^3$ given by $z (z_1, z_2) = (z^p z_1, z^q z_2)$, with $p,q\in  \mathbb Z$. The orbifold presented by the groupoid $S^1\ltimes S^3$ is called the {\em teardrop} orbifold if $p$ or $q$ are equal to 1 and the {\em football} orbifold otherwise.
If $gcd(p,q)=1$, the orbifold presented by $S^1\ltimes S^3$ is not developable. In this case, the isotropy group of $(z_1, 0)$ is $\mathbb Z_p$ and the one of $(0, z_2)$ is $\mathbb Z_q$. All other points have trivial isotropy.
If $p=q$, the orbifold presented by $S^1\ltimes S^3$ is Morita equivalent to $\mathbb Z_p\ltimes S^2$ and therefore it is developable. 
\end{exam}

\section{Connectedness}
We start in this section the comparative study of different invariants for group actions and for orbifolds. 

 \subsection{Connectedness for group actions}
\begin{defn}[$G$-connectedness]\cite{May2}
A $G$-space is {\em $G$-connected} if each of the $H$-fixed point spaces $X^H$ is non-empty and connected for every closed subgroup $H$ of
$G$.
\end{defn}

\begin{exam}[The $(p,q)$ action]
Consider the previous action of  $S^1$ on $S^3$ given by
$z (z_1, z_2) = (z^p z_1, z^q z_2)$. This space is not $S^1$-connected since the global fixed point set ${(S^3)}^{S^1}$ is empty for all values of $p,q\in  \mathbb Z$.
\end{exam}

\begin{remark}
$G$-connectedness is not invariant under Morita equivalence as is shown in the following example.
\end{remark}

\begin{exam}[]
Consider the previous action of $S^1$ on the solid torus $T_1$. There are no global fixed points. The core circle has isotropy $\Z_p$ and all the other points have trivial isotropy. That is, ${T_1}^{S^1}=\emptyset , {T_1}^{\Z_p}=S^1$. Hence $S^1\ltimes T_1$ is not $S^1$-connected.

Consider the action of $\Z_p$ on a disk $D$ by rotation. This action has a global fixed point $x=x_0$, $D^{\Z_p}=x_0$ which is connected. Therefore $\Z_p\ltimes D$ is  $\Z_p$-connected.

The groupoid $S^1\ltimes T_1$ is Morita equivalent to $\Z_p\ltimes D$, so $G$-connectedness is not an invariant.
\end{exam}

\subsection{Connectedness for orbifolds}

\begin{defn}[Groupoid connectedness] A groupoid $G\ltimes X$ is {\em connected} if there is a generalized path $I\overset{\e}{\gets}I'\overset{\alpha}{\to}G\ltimes X$ from any object to any other in $X$.
    \end{defn}  
    
    \begin{exam}[The $(p,q)$-action]
We will show that the groupoid $S^1\ltimes S^3$ given by the $(p,q)$-action when $p=q$,  is connected.  Let $D_i$ be a transversal disk to the orbits in each of the solid tori, we have that the transveral $D=D_1\cup D_2$ is connected. Given $(z_1, z_2)$ and $(z'_1, z'_2)$ in $S^3$ consider a point in the intersection of their orbits with the transversal: $x_1\in S^1{(z_1, z_2)}\cap D$ and $x_2\in S^1{(z'_1, z'_2)}\cap D$. Then there exists $z, z'\in S^1$ such that $z (z_1, z_2)=x_1$ and $z' (z'_1, z'_2)=x_2$. Consider a path $\sigma$ in $D$ from $x_1$ to $x_2$. We have that $\alpha=(z,\sigma, z')$ is a generalized path from $(z_1, z_2)$ and $(z'_1, z'_2)$.

Then the groupoid $S^1\ltimes S^3$ is connected as a groupoid whereas $S^3$ is not $S^1$-connected considered as an $S^1$-space. 
\end{exam}

The following theorem is attributed to Montgomery and Yang in \cite{Bredon}.
\begin{thm}[Chapter II, Theorem 6.2]\label{MY}
Let $X$ be a $G$-space, $G$ compact Lie, and let $f: I \rightarrow X/G$ be any path. Then there exists a lifting $f': I \rightarrow X$, $\pi_X  \circ f' =f$.
\end{thm}

We will prove that groupoid connectedness is the same as connectedness of the quotient space.

\begin{thm} If $G$ is a compact Lie group, then the groupoid $G\ltimes X$ is connected if and only if $X/G$ is connected.
\end{thm}

\begin{proof}
The quotient map $X \rightarrow X/G$ is surjective and a generalized path gives an honest path in the quotient space, therefore if any two points in $X$ can be connected by a generalized path then the quotient space is path connected. 

Now if the quotient space is path connected, take any two points in $X$, consider the image in the quotient space, by hypothesis there is a path in the quotient space joining the images, by the lifting theorem of Montgomery and Yang \ref{MY} there is a lift of this path. This lift does not necessarily start and end at the two points, but at elements in their orbits. From this we can construct a generalized path between the two original points.

\end{proof}

\begin{cor}
Connectedness is invariant under Morita equivalence. 
\end{cor}
Therefore, connectedness is well defined for orbifolds.

\begin{defn}
An orbifold $\X$ is {\em connected} if it is presented by a connected groupoid.
\end{defn}

\begin{exam}[The $(p,q)$-action]
The football orbifold  when $p=q$,  is connected.  \end{exam}

\section{Generalized maps between G-spaces}
In \cite{MM}, Moerdijk and Mr\v{c}un give a description of the correspondence between isomorphism classes of generalized maps and Hilsum-Skandalis maps.
We present in this section the theory of Hilsum-Skandalis maps \cite{M} adapted to the case of translation groupoids. This correspondence will allow us to establish a bridge between the orbifold world and the equivariant one by stating an explicit condition that determines when an orbifold map correlates to a $G$-map. 

\subsection{Right and left actions of translation groupoids on spaces}
Let $G\ltimes X$ be a topological groupoid and let
$E$ be a topological space. A {\em left $(G\ltimes X)$-action}
on $E$ is a pair $(\mu,p)$, where $p:E\to X$ and 
$\mu:G\times E \to E$ are continuous maps satisfying:
\begin{enumerate}
\item    $p(g\cdot e)=gx$, 
\item   $1\cdot e=e$
\item $g\cdot (g'\cdot e)= gg' \cdot e$
\end{enumerate}

for any $e\in E$ and $g,g'\in G$ with $p(e)=x$. 

We have a new translation groupoid $G\ltimes E$ where source and target $G \times E \rightrightarrows E$ are given by the $(G\ltimes X)$-action: $s(g,e)=e$ and $t(g,e)=g\cdot e$.

Analogously, given a  topological groupoid $H\ltimes Y$, we can define a {\em right $(H\ltimes Y)$-action}
on $E$ as a pair $(\nu,w)$, where $w:E\to Y$ and 
$\nu:E\times H \to E$ are continuous maps satisfying analogous conditions.

In this case, we have a translation groupoid $E\rtimes H$ where source and target are given by the $(H\ltimes Y)$-action: $s(e,h)=e\cdot h$ and $t(e,h)=e$.

\subsection{Principal $(G\ltimes X)(H\ltimes Y)$-bibundles}
A bibundle is a topological space $E$, 
$$\begin{tikzcd}
    &E \arrow{dr}{p}  \arrow[swap]{dl}{w} &  \\
     Y&& X
  \end{tikzcd}
$$
equipped with a left $(G\ltimes X)$-action
$(\mu,p)$ and a right $(H\ltimes Y)$-action $(\nu,w)$ such that the groupoids act along the fibers: $p(eh)=p(e)$ and $w(ge)=w(e)$.

The $(G\ltimes X)(H\ltimes Y)$-bibundle is {\em principal} if $w:E\to Y$ admits local sections and the map $(\mu, p_2):G\times E \to E\times_Y E$ is a homeomorphism.

The double translation groupoid $G\ltimes E\rtimes H$ is given by $G \times E \times H \rightrightarrows E$ where  $s(g,e,h)=e\cdot h$ and $t(g,e,h)=geh^{-1}$.

\subsection{Hilsum-Skandalis maps}
Two principal $(G\ltimes X)(H\ltimes Y)$-bibundles $(E,w,p)$ and $(E',w',p')$ are called {\em isomorphic} if there
exists a homeomorphism $E\to E'$ that intertwines the maps $w, p$ with the maps $w', p'$.

A {\em Hilsum-Skandalis map} from $H\ltimes Y $ to $G\ltimes X$ is an isomorphism class of 
principal $(G\ltimes X)(H\ltimes Y)$-bibundles.  (We will not make any distinction in notation between the principal bibundles and their isomorphism classes - the associated Hilsum-Skandalis maps.)

The {\em composition} of Hilsum-Skandalis maps is given by:
 
$$\begin{tikzcd}
    &E \arrow{dr}{p}  \arrow[swap]{dl}{w} &  \\
     Y&& X
  \end{tikzcd}
  \circ
\begin{tikzcd}
    &E' \arrow{dr}{p'}  \arrow[swap]{dl}{w'} &  \\
     Z&& Y
  \end{tikzcd}
  =
\begin{tikzcd}
    &E^* \arrow{dr}{p^*}  \arrow[swap]{dl}{w^*} &  \\
     Z&& X
  \end{tikzcd}
$$

where $E^*$ is the orbit space of the following right $(H\ltimes Y)$-action on $E\times_YE'$:

$$w''=wp_1: E\times_YE \to Y, (e,e')\mapsto w(e)$$ 
$$\nu'': E\times_YE\times H \to E\times_YE, ((e,e'), h)\mapsto (eh, h^{-1}e').$$

The maps $w^*:E^* \to Z$ and $p^*: E^*\to X$ are given by $w^*([e,e'])=w'(e')$ and 
$p^*([e,e'])=p(e)$.

The left $(G\ltimes X)$-action is given by $\mu^*(g,[e,e'])=[ge,e']$ and the right 
$(K\ltimes Z)$-action is given by $\nu^*([e,e'], k)=[e,e'k]$.

\subsection{From Hilsum-Skandalis maps to generalized maps}
Given the Hilsum-Skandalis map 
$$\begin{tikzcd}
    &E \arrow{dr}{p}  \arrow[swap]{dl}{w} &  \\
     Y&& X
  \end{tikzcd}
$$
the associated generalized map is 
$$H\ltimes Y\overset{\mathfrak{w}}{\gets}G\ltimes E\rtimes H\overset{\mathfrak{p}}{\rightarrow}G\ltimes X$$
where $G\ltimes E\rtimes H$ is given by $G \times E \times H \rightrightarrows E$ with $s(g,e,h)=e$ and $t(g,e,h)=geh^{-1}$. The maps $\mathfrak{w}$ and $\mathfrak{p}$ are given by $\mathfrak{p}_0=p$ and $\mathfrak{p}_1(g,e,h)=(g,p(e))$ and $\mathfrak{w}_0=w$ and $\mathfrak{w}_1(g,e,h)=(h,h^{-1}w(e))$.
\subsection{From generalized maps to Hilsum-Skandalis maps}
Given an equivariant map $f\ltimes \phi : K\ltimes Z \to G\ltimes X$ we associate the Hilsum-Skandalis map $\langle\phi\rangle$ defined by:
$$\begin{tikzcd}
    &G\times Z \arrow{dr}{p}  \arrow[swap]{dl}{w} &  \\
     Z&& X
  \end{tikzcd}
$$
with $w(g,z)=z$ and $p(g,z)=g\phi(z)$. The actions are given by $\mu(g',g,z)=(gg',z)$ and $\nu(g,z,k)=(gf(k), k^{-1}z)$.

Given the essential equivalence $\sigma\ltimes \epsilon : K\ltimes Z \to H\ltimes Y$, there is a Hilsum-Skandalis map $\langle\epsilon\rangle^{-1}$ such that the composition $\langle\epsilon\rangle \circ \langle\epsilon\rangle^{-1}$ is the identity.
The map $\langle\epsilon\rangle^{-1}$ is defined by:
$$\begin{tikzcd}
    &H\times Z \arrow{dr}{p'}  \arrow[swap]{dl}{w'} &  \\
     Y&& Z
  \end{tikzcd}
$$
with $w'(h,z)=h\epsilon(z)$ and $p'(h,z)=z$. The actions are given by $\mu'(k,h,z)=(h\sigma(k^{-1}),kz)$ and $\nu'(h,z,h')=({h'}^{-1}h,z)$.

Given a generalized map $$H\ltimes Y\overset{\sigma\ltimes \epsilon}{\gets}K\ltimes Z\overset{f\ltimes \phi}{\rightarrow}G\ltimes X$$ the associated Hilsum-Skandalis map will be given by the composition $\langle\phi\rangle \circ \langle\epsilon\rangle^{-1}$:
$$\begin{tikzcd}
    &^{G\times H\times Z}\!/\!_{K\ltimes Z} \arrow{dr}{p''}  \arrow[swap]{dl}{w''} &  \\
     Y&& X
  \end{tikzcd}
$$
where $w''([g,h,z])=h\epsilon(z)$ and $p''([g,h,z])=g\phi(z)$. The classes of equivalence for the right $(K\ltimes Z)$-action on $G\times H\times Z$ are given by:
$(g,h,z)\sim (gf(k),h\sigma(k),k^{-1}z)$ for all $k\in K$.

The left $(G\ltimes X)$-action is given by $\mu''(g',[g,h,z])=[g'g,h,z]$ and the right 
$(H\ltimes Y)$-action is given by $\nu''([g,h,z], h')=[g,(h')^{-1}h,z]$.

\subsection{Characterization of equivariant maps}
We will show in this section that all generalized maps of a certain type are equivalent to a $G$-map.

\begin{prop} \label{strict} The generalized map $H\ltimes Y\overset{\sigma\ltimes \epsilon}{\gets}K\ltimes Z\overset{f\ltimes \phi}{\rightarrow}G\ltimes X$ is equivalent to an equivariant map $H\ltimes Y\overset{}{\rightarrow}G\ltimes X$ if and only if the map $w: {G\times H\times Z}\!/\!_{K\ltimes Z} \to Y$ with $w([g,h,z])=h\epsilon(z)$ has a global section.
\end{prop}
\proof The result follows from proposition II.1.5 in \cite{M} and the previous correspondence between generalized maps and Hilsum-Skandalis maps for translation groupoids.

\begin{defn} If $K$ acts on $Z$, we say that two homomorphisms $m: K \rightarrow G$ and $n: K \rightarrow G$ are \textit{naturally conjugate}, if
there exists a continuous function, $\lambda :  Z \rightarrow G$ such that for every $k \in K,z \in Z$ we have  $n(k) =\lambda(  k z )m(k)\lambda (  z )^{-1}$. In particular if $\lambda$ is constant we have that there exists $\gamma\in G$ such that $n(k)=\gamma m(k) \gamma^{-1}$.
\end{defn}

\begin{thm}\label{section}
Any generalized map $G\ltimes Y\overset{m\ltimes \epsilon}{\gets}K\ltimes Z\overset{n\ltimes \phi}{\rightarrow}G\ltimes X$ with $m$ and $n$ \textit{naturally conjugate} is equivalent to a $G$-map by a natural transformation.
\end{thm}

\proof 

Following remark \ref{pronk} we factor the  generalized map $G\ltimes Y\overset{m\ltimes \epsilon}{\gets}K\ltimes Z\overset{n\ltimes \phi}{\rightarrow}G\ltimes X$ as 
$$G\ltimes \left(G\times_{K/L}{Z/L}\right)\overset{a}{\gets}K/L\ltimes Z/L\overset{b}{\gets}K\ltimes Z\overset{n\ltimes \phi}{\rightarrow}G\ltimes X$$
where $L$ is a normal subgroup of $K$ acting freely on $Z$ and $a\circ b=m\ltimes \epsilon$. We have that $m: K \to K/L \subset G$ is the quotient projection followed by the inclusion, $m(k)=\bar k$ and $\epsilon: Z \to G\times_{K/L}{Z/L}$ is the projection on the orbit space followed by the inclusion in the induced product, $\epsilon(z)=[e,[z]]$.

Since $m$ and $n$ are naturally conjugate there is a function $\lambda: Z \rightarrow G\times_{K/L}{Z/L}$ such that $n(k) =\lambda(  k z )\bar k\lambda (  z )^{-1}$ for all  $k \in K, z \in Z$.

The right principal bundle of its associated Hilsum-Skandalis map is:
$$w: {G\times G\times Z}\!/\!_{K\ltimes Z} \to G\times_{K/L}{Z/L}$$ with $w([g,h,z])=[h,[z]].$

The classes of equivalence are given by:
$(g,h,z)\sim (gn(k) ,hm(k),k^{-1}z)$ for all $k\in K$ and $z\sim l z$ for all $l\in L$. Since $n(k) =\lambda(  k z )\bar k\lambda (  z )^{-1}$, we have
$$(g,h,z)\sim (g\lambda(  k z )\bar  k\lambda (  z )^{-1} ,h \bar k ,k^{-1}z).
$$

We will see that $w$ has a global section. Define $\tau: G\times_{K/L}{Z/L} \to  {G\times G\times Z}\!/\!_{K\ltimes Z}$ as $\tau([h,[z]])=[h\lambda(z)^{-1},h,z]$.
To prove that is well defined, we first calculate $\tau([hk^{-1},[klz]])$ for $k\in K, l\in L$. We have that 
$\tau([hk^{-1},[klz]])=[ hk^{-1} \lambda(klz)^{-1},hk^{-1},klz]$. By the equivalence relation is the same as $$ [  hk^{-1}  \lambda(klz)^{-1} \lambda(  kl z )  \overline{kl}\lambda (  z )^{-1},hk^{-1}\overline{kl},{(kl)}^{-1}klz]=[  hk^{-1}\overline{kl} \lambda(z)^{-1},h\bar l,z]$$ 
Since $\bar l =e$, we have that 
$$[  h\bar l \lambda(z)^{-1},h\bar l,z]=[h\lambda(z)^{-1},h,z]=\tau([h,z]).$$ Moreover, $\tau$ is a global section since $w\tau([h,[z]])=w([h\lambda(z)^{-1},h,z])=[h,[z]]$.

We have proved that the map $w: {G\times G\times Z}\!/\!_{K\ltimes Z} \to G\times_{K/L}{Z/L}$ has a global section. From proposition \ref{strict}, the generalized map $G\ltimes Y\overset{m\ltimes \epsilon}{\gets}K\ltimes Z\overset{n\ltimes \phi}{\rightarrow}G\ltimes X$ is then equivalent to an equivariant map.

Moreover, the global section determines the explicit expression of the equivariant map $ \xi\times \varphi: G\ltimes Y \to G\ltimes X$. We will see that the homomorphism $\xi$ between the groups is the identity and the equivariant map is actually a $G$-map. 

Define $\varphi= p\tau$ on objects where 
$$p:{G\times G\times Z}\!/\!_{K\ltimes Z} \to Y$$ is the other leg of the bibundle given by $p([g,h,z])=g\phi(z)$.  Since $w$ is principal, the equivariant map is uniquely defined on arrows.

We have that the map $\varphi: G\times_{K/L}{Z/L} \to  X$ is given by

$$\varphi([h,[z]])=p\tau([h,[z]])=p([  h \lambda(z)^{-1},h,z])=  h \lambda(z)^{-1}\phi(z)$$ and it is the identity on arrows since 
$$ g \varphi([h,[z]])=gh \lambda(z)^{-1}\phi(z)=\varphi([gh,[z]])=\varphi(g[h,[z]]).$$ \qed

We have constructed a $G$-map $\varphi :  Y \rightarrow X$ that makes the following diagram commutative up to equivalence:
$$\xymatrix{
K\ltimes Z \ar[r]^{n \ltimes \phi} \ar[d]^{m \ltimes \e}
&G\ltimes X  \\ G\ltimes Y \ar@{.>}[ur]_{\id \ltimes \varphi} & }
$$
by taking  $y\in Y$, finding $g \in G$ and $z \in Z$, such that $g \e(z) =y$, and defining:
$$
\varphi(y) = g \lambda(z)^{-1} \phi(z).
$$
Note that by construction $(\id \ltimes \varphi)\circ (m \ltimes \e)$ and $n \ltimes \phi$  are related by the natural transformation $\lambda$, i.e $ \lambda(z) \varphi(\e(z))=\phi(z)$ and $n(k) =\lambda(  k z )m(k)\lambda (  z )^{-1}$.

\begin{remark}
Note that without the hypothesis of conjugation, we can have generalized maps  $G\ltimes Y\overset{m\ltimes \epsilon}{\gets}K\ltimes Z\overset{n\ltimes \phi}{\rightarrow}G\ltimes X$ that are not equivalent to an equivariant map.
\end{remark}

\begin{exam}
Consider the following actions on the circle $S^1$. Let $\Z_2 \times \Z_2=\langle \rho, \sigma \rangle$  acting on $S^1$  by  rotation $\rho$ and  reflection $\sigma$ and $\Z_2=\langle \rho, \sigma \rangle/\langle \rho\rangle=\langle \sigma \rangle$  acting on $S^1$  by  reflection.

We have an essential equivalence $m\ltimes \epsilon : (\Z_2 \times \Z_2)\ltimes S^1 \to \Z_2 \ltimes S^1$ with $m(a,b)=b$ and $\epsilon(z)=z^2$ and an equivariant map $n\ltimes \phi : (\Z_2 \times \Z_2)\ltimes S^1 \to \Z_2 \ltimes S^1$ with $n(a,b)=a$ and $\phi(z)=z^2$. The map $w: {\Z_2 \times \Z_2\times S^1}\!/\!_{(\Z_2 \times \Z_2)\ltimes S^1} \to S^1$ is given by $w([g,h,z])=hz^2$. We observe that ${\Z_2 \times \Z_2\times S^1}\!/\!_{(\Z_2 \times \Z_2)\ltimes S^1}$ where
$$(g,h,z)\sim (ga, hb, (a,b)z)$$
 is homeomorphic to $S^1$ via the map $\delta: {\Z_2 \times \Z_2\times S^1}\!/\!_{(\Z_2 \times \Z_2)\ltimes S^1} \to S^1$ given by $\delta([g,h,z])=(g,h)z$ and $\delta^{-1}(z)=[e,e,z]$. The map $w: S^1 \to S^1$ is the double covering of the circle, $w(z)=z^2$ and does not have a global section. Therefore, the generalized map
 $$\Z_2 \ltimes S^1\overset{m\ltimes \epsilon}{\gets}(\Z_2 \times \Z_2)\ltimes S^1\overset{n\ltimes \phi}{\rightarrow}\Z_2 \ltimes S^1$$ is not equivalent to an equivariant map.
 
 Observe that in this case $m$ and $n$ are not naturally conjugate: since $S^1$ is connected and $\Z_2$ is abelian, we have that in this case the homomorphisms
$m$ and $n$ are naturally conjugate if and only if $m=n$.

\end{exam}

\section{Lusternik-Schnirelmann category}
We prove in this section that the $G$-category introduced by Fadell coincides with our orbifold category for a discrete group $G$ acting continuously on a Hausdorff space $X$.
\subsection{G-category}
An open set $U\subseteq X$ is  $G$-{\em invariant} if $gU\subseteq U$ for all $g\in G$. 

\begin{defn} We say that a $G$-invariant subset $U \subseteq X$ is {\em $G$-compressible into a $G$-invariant subset $A \subseteq X$}, 
if the inclusion map $i_U : U \rightarrow X$ is $G$-homotopic to a $G$-map $c: U \rightarrow X$ with $c(U) \subseteq A$.
\end{defn}

 \begin{defn} A $G$-invariant subset $U \subseteq X$  is called {\em $G$-categorical} if $U$ is $G$-compressible into a single orbit. 
\end{defn}
A space $X$ is {\em $G$-contractible} if $X$ is $G$-categorical. We now describe the equivariant category of a $G$-space $X$ introduced by Fadell in \cite{F}.

\begin{defn} The {\em equivariant category} of a $G$-space $X$, denoted $\gcat(X)$, is the least integer $k$ such that $X$ may be covered by $k$ open sets $\{ U_1,\ldots , U_k\}$, each of which is $G$-categorical.
\end{defn}

\subsection{grd-category}
In the rest of this paper $G$ denotes a discrete group.
\begin{defn} We say that a $G$-invariant subset $U \subseteq X$ is {\em grd-compressible into a $G$-invariant subset $A \subseteq X$}, 
if the inclusion map $i_U : G\ltimes U \rightarrow G\ltimes X$ is grd-homotopic to a map $c: G\ltimes U \rightarrow G\ltimes X$ with $c(U) \subseteq A$.
\end{defn}

 \begin{defn} A $G$-invariant subset $U \subseteq X$  is called {\em grd-categorical} if $U$ is grd-compressible into a single orbit. 
\end{defn}
We say that the groupoid $G\ltimes X$ is {\em grd-contractible} if $X$ is grd-categorical. We now introduce the groupoid category of a groupoid $G\ltimes X$.

\begin{defn} The {\em groupoid category} of groupoid $G\ltimes X$, denoted $\cat_{grd}(G\ltimes X)$, is the least integer $k$ such that $X$ may be covered by $k$ open sets $\{ U_1,\ldots , U_k\}$, each of which is grd-categorical.\end{defn}

\begin{thm}
Consider the groupoid $G\ltimes X$ given by the action of the group $G$ on $X$. Then $\cat_{grd}(G\ltimes X)=\gcat(X)$.
\end{thm}

\begin{proof}

We see first that $\gcat(X)\ge\cat_{grd}(G\ltimes X)$. Consider a $G$-invariant subset $U \subseteq X$ which is $G$-categorical. We have that there is a $G$-homotopy $H$ between the inclusion map $i_U : U \rightarrow X$ and a $G$-map $c: U \rightarrow X$ with $c(U) $ contained in a single orbit. Since a $G$-homotopy is also a groupoid homotopy \cite[Proposition 6.13]{AC}, we have that $U$  is  grd-categorical.

To prove that $\gcat(X)\le\cat_{grd}(G\ltimes X)$, we consider a $G$-invariant subset $U \subseteq X$ such that 
the inclusion map $i_U : G\ltimes U \rightarrow G\ltimes X$ is grd-homotopic to a map $c: G\ltimes U \rightarrow G\ltimes X$ with $c(U)$ contained in a single orbit. This means we have generalized map
$G\ltimes U\overset{\varepsilon}{\gets}\tilde G\ltimes \tilde U\overset{\tilde H}{\to}G\ltimes X^I$
that makes the following diagram commutes up to 2-isomorphisms,
$$
\xymatrix{ G\ltimes X&&
{G \ltimes X^I}\ar[rr]^{\ev_{1}}="0" \ar[ll]_{\ev_0 }="2"&&G\ltimes X\\
&{}&{\tilde G\ltimes \tilde U} \ar[u]_{\tilde H} \ar[d]^{\varepsilon}&{}&\\
&&{G\ltimes U}\ar[uull]^{i_U}\ar[uurr]_{c}&&
}
$$
We will prove that there is a $G$-homotopy $H$ between the inclusion map $i_U : U \rightarrow X$ and a $G$-map $c: U \rightarrow X$ with image in an orbit.

The left diagram commutes up to 2-isomorphism, meaning there is $\widehat{G} \ltimes \widehat{U}$ and an essential equivalence $\nu :\widehat{G} \ltimes \widehat{U} \rightarrow \tilde G\ltimes \tilde U $, such that the following diagram commutes up to natural transformation,

$$
\xymatrix{ G\ltimes X&&
{G \ltimes X^I} \ar[ll]_{\ev_0 }="2"&&\\
&\widehat{G}\ltimes \widehat{U} \ar[r]_{\nu}&{\tilde G\ltimes \tilde U} \ar[u]_{\tilde H} \ar[d]^{\varepsilon}&{}\\
&&{G\ltimes U}  \ar@/^2.0pc/[uull]^{i_U}&
}
$$
i.e the following diagram commutes up to natural transformation,
$$
\xymatrix{ G\ltimes X&&
{G \ltimes X^I} \ar[ll]_{\ev_0 }="2"&&\\
&{}&{\widehat{G}\ltimes \widehat{U}} \ar[u]_{\tilde H \circ \nu} \ar[d]^{ \varepsilon\circ \nu }&{}\\
&&{G\ltimes U}\ar[uull]^{i_U}&
}
$$
Let us call $n \ltimes \phi = \tilde H \circ \nu$ and $m \ltimes \e =\varepsilon\circ \nu$. The commutativity up to natural transformation $ \ev_0  \phi \sim \e$, means that there exists a natural transformation, $\lambda : \widehat{U} \rightarrow G$ such that for every $\hat u$ we have $\lambda(\hat u) \ev_0  \phi(\hat u) = \e(\hat u)$. Naturality means that $\lambda( k \hat u ) m(k) = n(k) \lambda ( \hat u )$, which can be re-written as $n(k) =\lambda(  k \hat u )m(k)\lambda ( \hat u )^{-1} \forall k,\hat u$. This is precisely the hypothesis of theorem \ref{section}, which gives a map $H$, such that the following diagram commutes up to the natural transformation $\lambda$:
$$\xymatrix{
\widehat{G}\ltimes \widehat{U} \ar[r]^{\tilde H \circ \nu} \ar[d]^{\varepsilon \circ \nu }
&G\ltimes X^I  \\ G\ltimes U \ar@{.>}[ur]_{H} & }
$$

Now we check that the new homotopy is from the inclusion to a map that lands in an orbit.

By construction of $H$ we are finding $\hat u \in \widehat{U}$ and $ g \in  G$ such that $ g \e(\hat u)=u$, and then,
$H (u) =  g \lambda(  \hat u )^{-1}\tilde H  (\nu ( \hat  u)).$
Evaluating at $t=0$, we have
$$
ev_0(H (u)) =  g\lambda(  \hat u )^{-1} ev_0( \tilde H ( \nu ( \hat  u))) 
$$
which is equal to
$$g \lambda(  \tilde u )^{-1}\lambda(\tilde u) i_U (\e(\hat u)) =  gi_U ( \e(\hat u)) = i_U ( g\e(\hat u)) =  i_U( u)$$
by the commutativity up to natural transformation $\lambda$ and the fact that $i_U$ is $G$-equivariant.

To evaluate at $t=1$, we have that there is a groupoid $\widehat{\widehat{G}} \ltimes \widehat{\widehat{U}}$ and an essential equivalence $\eta :\widehat{\widehat{G}} \ltimes \widehat{\widehat{U}} \rightarrow \tilde G\ltimes \tilde U $  that makes the following diagram commute up to natural transformations 
$$
\xymatrix{ &&
{G \ltimes X^I}\ar[rr]^{\ev_{1}}="0" &&G\ltimes X\\
&\widehat{G}\ltimes \widehat{U} \ar[r]_{\nu}&{\tilde G\ltimes \tilde U} \ar[u]_{\tilde H } \ar[d]^{\varepsilon}& \widehat{\widehat{G}} \ltimes \widehat{\widehat{U}}  \ar[l]^{\eta}&\\
&&{G\ltimes U}  \ar@/_2.0pc/[uurr]_{c} \ar@/^8.0pc/[uu]^H&&
}
$$
since $c$ is an equivariant map and the right part of the diagram commutes up to 2-isomorphism. 

Let us see that the homotopy lands in an orbit. By definition of  $H$, for $u \in U$, we find $ g \in G$, $\hat u \in \widehat{U}$ with $g\e(\hat u) = u$, where $m \ltimes \e =\varepsilon\circ \nu$. Since $\eta$ is also an essential equivalence, there exists $\tilde g \in \widetilde{G}$ and $\widehat{\widehat{u}} \in \widehat{\widehat{U}}$, such that  $\tilde g \eta(\widehat{\widehat{u}})=\nu(\hat u)$. The commutativity of the right part of the diagram up to natural transformation says that, for some $\Lambda : \widehat{\widehat{U}} \rightarrow G$, we have
$$
\Lambda(\widehat{\widehat{u}})ev_1(\tilde H(\eta(\widehat{\widehat{u}}))) = c(\varepsilon(\eta(\widehat{\widehat{u}}))).
$$
By definition of $H$, we have
$$
ev_1(H (u)) =  g \lambda(  \hat u )^{-1} ev_1(\tilde H ( \nu(\hat  u))). 
$$
Using that $\tilde H$ is equivariant with respect to some homomorphism $l : \tilde G \rightarrow G$, we have:
$$
ev_1(H (u)) =  g  \lambda(  \hat u )^{-1}ev_1(\tilde H ( \tilde g \eta(\widehat{\widehat{u}}))) = g \lambda(  \hat u )^{-1}l(\tilde g) ev_1(\tilde H (\eta(\widehat{\widehat{u}})))
$$
which is equal to
$g \lambda(  \hat u )^{-1}l(\tilde g)\Lambda(\widehat{\widehat{u}})^{-1} c(\varepsilon(\eta(\widehat{\widehat{u}})))$.
But since $c : U \rightarrow X$ lands in an orbit, then $g \lambda(  \hat u )^{-1}l(\tilde g)\Lambda(\widehat{\widehat{u}})^{-1} c(\varepsilon(\eta(\widehat{\widehat{u}})))$ too. 
\end{proof}

The equivariant category is invariant under Morita equivalence for compact Lie group actions on metrizable spaces as shown in \cite{morita}.

\begin{cor}
The groupoid category is invariant under Morita equivalence for translation groupoids $G\ltimes X$ where $G$ is a finite group and $X$ a metrizable space.
\end{cor}
Therefore, groupoid category is well defined for orbifolds given by the action of finite group. 

\begin{defn} Let $G$ be a finite group acting on $X$. The {\em orbifold category} of the orbifold $\X$ presented by the groupoid $G\ltimes X$, denoted $\cat_{orb}(\X)$, is the groupoid category of $G\ltimes X$, i.e. $\cat_{orb}(\X)=\cat_{grd}(G\ltimes X)$.
\end{defn}

\bibliography{AngelColman}
\bibliographystyle{abbrv}
\end{document}